\documentclass[11pt, reqno]{amsart}
\usepackage{amsmath,amssymb,amsfonts,amsthm,graphicx,mathrsfs,mathtools}
\usepackage[margin=1.5in]{geometry}
\usepackage[dvipsnames]{xcolor}
\makeatletter
\@namedef{subjclassname@2020}{%
\textup{2020} Mathematics Subject Classification}
\makeatother

\title[Minimal surfaces in $\mathbb{R}^4$ like the Lagrangian catenoid]{Minimal surfaces in $\mathbb{R}^4$ like the Lagrangian Catenoid}

\author[J. Lee]{Jaehoon Lee}
\address[]{Jaehoon Lee, Department of Mathematical Sciences, Seoul National University, Seoul  08826, Korea}
\email{jaehoon.lee@snu.ac.kr}

\begin{document}

\newtheorem{theorem}{theorem}[section]
\newtheorem{thm}[theorem]{Theorem}
\newtheorem{lem}[theorem]{Lemma}
\newtheorem{cor}[theorem]{Corollary}
\newtheorem{prop}[theorem]{Proposition}
\newtheorem{rmk}[theorem]{Remark}
\newtheorem{Def}[theorem]{Definition}
\newtheorem{Question}[theorem]{Question}
\newtheorem{Ex}[theorem]{Example}

\renewcommand{\theequation}{\thesection.\arabic{equation}}
\newcommand{\RNum}[1]{\uppercase\expandafter{\romannumeral #1\relax}}
\newcommand{\R}{\mathbb{R}}
\newcommand{\C}{\mathbb{C}}
\newcommand{\grad}{\nabla}
\newcommand{\laplacian}{\Delta}
\renewcommand{\d}{\textup{d}}


\begin{abstract}
In this paper, we discuss complete minimal immersions in $\mathbb{R}^N$($N\geq4$) with finite total curvature and embedded planar ends. First, we prove nonexistence for the following cases: (1) genus 1 with 2 embedded planar ends, (2) genus $\neq4$, hyperelliptic with 2 embedded planar ends like the Lagrangian catenoid. Then we show the existence of embedded minimal spheres in $\R^4$ with 3 embedded planar ends. Moreover, we construct genus $g$ examples in $\R^4$ with $d$ embedded planar ends such that $g\geq 1$ and $g+2\leq d\leq 2g+1$. These examples include a family of embedded minimal tori with 3 embedded planar ends.
\end{abstract}

\maketitle

\section{\textbf{Introduction}}
\setcounter{equation}{0}
The Lagrangian catenoid in $\R^4$ can be identified with
\begin{align}\label{LCat}
\left\{(z, \tfrac{1}{z})\in \C^2\ |\ z\in \C-\{0\}\right\}
\end{align}
up to rigid motions and scaling. It is one of the special Lagrangian surfaces constructed by Harvey and Lawson \cite{HL}. Later, Castro and Urbano \cite{CU} characterized it as a unique special Lagrangian surface foliated by circles. Moreover, it is asymptotic to 2 planes and has total curvature $-4\pi$.

Since the Lagrangian catenoid is given by the graph of a holomorphic function, it becomes a minimal surface in $\R^4$. It is also an important example among minimal surfaces. Indeed, it is one of the complete doubly-connected minimal immersions in $\R^N$ with total curvature $-4\pi$ classified by Hoffman and Osserman \cite{HO1}. Furthermore, Park \cite{Park} proved that a circle-foliated minimal immersion in $\R^N$ not contained in any 3-dimensional Affine subspaces is a part of the Lagrangian catenoid.

In this paper, we are interested in the asymptotic behavior of the Lagrangian catenoid. Minimal surfaces with embedded planar ends have received a lot of attention for many reasons. They not only provide interesting examples in the theory of minimal surfaces but are also related to Willmore surfaces in the round sphere. Indeed, Bryant \cite{Br1} first proved that minimal surfaces in $\R^3$ with embedded planar ends compactify to give Willmore surfaces in $\mathbb{S}^3$, and conversely, all Willmore spheres are given in this way. 

There are many results concerning existence and nonexistence in $\R^3$, and here we recall some of them. For nonexistence, Bryant \cite{Br2} showed that there are no complete minimal spheres with 3, 5, and 7 embedded planar ends. Using the spinor representation, Kusner and Schmitt \cite{Spin} proved nonexistence for complete minimal tori with 3 embedded planar ends. Also, many examples have been constructed: see \cite{Spin, Sham} and references therein.

It should also be mentioned that there are no examples in $\R^3$ with 2 embedded planar ends regardless of genus. This can be verified by applying the Kusner theorem (see \cite[Theorem~A]{K1}) and the maximum principle. Therefore examples like the Lagrangian catenoid are only possible in high codimension. Now it is natural to ask in high codimension whether there exist minimal surfaces with 2 embedded planar ends and positive genus.  

In this regard, we deal with nonexistence in Section \ref{MAIN}. We prove nonexistence for the following cases: (1) genus 1 minimal surfaces with 2 embedded planar ends (Theorem \ref{G1}), (2) genus $\neq4$, hyperelliptic minimal surfaces with 2 embedded planar ends like the Lagrangian catenoid (Theorem \ref{G2}). Note that our definition of the embedded planar end only allows the multiplicity to be one when the end approaches the asymptotic plane. This is because our main interest is on the Lagrangian catenoid, and it can be viewed as a direct generalization of the embedded planar end in $\R^3$. See Section \ref{EPE} for definitions and more details on planar ends.

Our proof uses the generalized Weierstrass representation. We take the special form of meromorphic differentials (see (\ref{GForm})) and use algebraic computations. The main difficulty in this method is the period problems. To avoid this, we also rely on intrinsic conformal symmetries. Here $\R^4$-version of the Kusner theorem (Corollary \ref{Kus4}) plays a significant role.

Finally, we construct several examples in $\R^4$ with more than 2 embedded planar ends in Section \ref{MAINN}. In the first part, the existence of embedded minimal spheres with 3 embedded planar ends is discussed. Then we construct genus $g$ complete minimal immersions in $\R^4$ with $d$ embedded planar ends for each $(g, d)$ satisfying $g\geq1$ and $g+2\leq d\leq2g+1$. These examples contain a family of embedded minimal tori with 3 embedded planar ends. Note that all examples are given as holomorphic curves in $\C^2$.

The paper is organized as follows: In Section \ref{EPE} we define the notion of embedded planar end in $\R^N$ in terms of the generalized Weierstrass representation. Moreover, $\R^4$-version of the Kusner theorem is discussed. In Section \ref{LEMM} we provide a special form of the Weierstrass data that we use in the proof. We also specify meromorphic 1-forms in several cases. Section \ref{MAIN} is devoted to the nonexistence of complete minimal immersions with 2 embedded planar ends. Finally, we construct examples with more than 2 embedded planar ends in Section \ref{MAINN}. 

Throughout the paper, we use the identification $\C^2\simeq\R^4$ via $(z, w)\in\C^2\leftrightarrow (\mbox{Re}z,\mbox{Im}z,\mbox{Re}w,\mbox{Im}w)\in\R^4$.

\section*{\textbf{Acknowledgements}}
The author would like to express his gratitude to Jaigyoung Choe for his thoughtful encouragement and for raising the question of nonexistence. This work was supported in part by NRF-2018R1A2B6004262.

\section{\textbf{Embedded planar ends}}\label{EPE}
\setcounter{equation}{0}  
\subsection{Embedded planar ends in arbitrary codimension}%\hfill\\
Let us consider a complete immersed minimal surface of finite total curvature in $\R^N$, $N\geq 3$. Up to translations, it can be represented by an immersion $X: \Sigma_g\backslash\{p_1, p_2, \cdots, p_k\}\to\R^N$, where
\begin{align}\label{WR}
X=\mbox{Re}\int\left(\phi_1, \phi_2, \cdots, \phi_N\right)
\end{align}
with the following conditions:
\begin{itemize}
\item $\Sigma_g$ is a closed Riemann surface of genus $g$, and $\{p_1, p_2, \cdots, p_k\}\subset\Sigma_g$,
\item $\phi_i$'s are meromorphic 1-forms on $\Sigma_g$ with possible poles at $\{p_1, p_2, \cdots, p_k\}$, and at least one of them should have a pole at each $p_i$,
\item $\sum_{i=1}^N\phi_i^2\equiv0$ and $\sum_{i=1}^N|\phi_i|^2>0$,
\item $\phi_i$'s have no real periods.
\end{itemize}
It is called \emph{the generalized Weierstrass representation} and we call $(\phi_1, \phi_2, \cdots, \phi_N)$ \emph{the Weierstrass data}. Note that each $p_i$ corresponds to a limit end of the surface. Moreover, we regard the map $\Phi\coloneqq[\phi_1, \phi_2, \cdots, \phi_N]: \Sigma_g \to \mathbb{P}^{N-1}$ as \emph{the generalized Gauss map}.

Now let us define the notion of embedded planar end in $\R^N$. We say that the minimal immersion $X$ has an \emph{\textbf{embedded planar end at $p_i$}} if there exists an open neighborhood $U_i$ of $p_i$ such that the image $X(U_i\backslash\{p_i\})$ is asymptotic to a plane and it is a graph over the same plane outside some compact subset, where all height functions have no logarithmic growth. The immersion $X$ is said to have embedded planar ends provided that it has an embedded planar end at each $p_i$.

This definition coincides with the classical notion of embedded planar end in $\R^3$, when $N=3$. However, in high codimension, the embeddedness of an end as a set does not guarantee that it is a graph over the asymptotic plane. Each end can be a multi-valued graph. Since we are mainly interested in multiplicity one ends like the Lagrangian catenoid, we rule out multiplicities by imposing the graph condition. For a future study, it will be an interesting problem to classify all possible types of embedded ends in $\R^{N\geq4}$ including multiplicities or logarithmic growth.

On the other hand, the definition can be expressed in terms of the Weierstrass data as follows:
\begin{prop}\label{EPW}
Let $X: \Sigma_g\backslash\{p_1, p_2, \cdots, p_k\}\to \R^N$ be a complete minimal immersion of finite total curvature with the Weierstrass data $(\phi_1, \phi_2, \cdots, \phi_N)$ as in the above. The immersion $X$ has embedded planar ends if and only if
\begin{align}\label{EPWD}
\min_{1\leq j\leq N}\emph{ord}_{p_i}\phi_j=-2\ \text{and}\ \emph{res}_{p_i}\phi_1=\emph{res}_{p_i}\phi_2=\cdots=\emph{res}_{p_i}\phi_N=0
\end{align}
for all $i=1, 2, \cdots, k$. 
\end{prop}
\begin{proof}
For $N=3$, we may refer the reader to \cite{Spin}. By our definition of the embedded planar end in $\R^N$, almost the same proof also applies to $N\geq4$.
\end{proof}

\begin{rmk}\label{RVD}
\textup{The image $X(\Sigma_g\backslash\{p_1, p_2, \cdots, p_k\})\subset\R^N$ is contained in some Affine subspace of dimension $\min\{2k, N\}$ provided that it has embedded planar ends. Indeed, following the notation in \cite{Br1}, consider $V_D=\{\omega\in H^0(\Sigma_g, C\otimes[2D])\ |\  \text{res}_p\omega=0\ \forall p\in D,\ \text{Re}(\text{Period}_{\gamma}\omega)=0\ \forall\gamma\in H_1(\Sigma_g, \mathbb{Z})\}$, where $D$ is a divisor in $\Sigma_g$ and $C$ is the canonical bundle. It is known as a real vector space of dimension $2\cdot\text{deg}D$. Now let $D=p_1+p_2+\cdots+p_k$. Then Proposition \ref{EPW} implies that meromorphic 1-forms in the Weierstrass data are the elements of $V_D$, which is of real dimension $2k$. Therefore the observation easily follows from (\ref{WR}).}
\end{rmk}

\begin{Ex}[the Lagrangian Catenoid in $\R^4$]\label{EX1}
\textup{Using the identification $\C^2\simeq\R^4$, we may write $X(z)=(z, \frac{1}{z})\in\C^2$ (see (\ref{LCat})) as 
\begin{align*}
X(z)=\left(\mbox{Re}z, \mbox{Im}z, \mbox{Re}\frac{1}{z}, \mbox{Im}\frac{1}{z}\right)\in\R^4,\ z\in\C-\{0\}, 
\end{align*}
and the Weierstrass data can be computed as follows: 
\begin{align*}
(\phi_1, \phi_2, \phi_3, \phi_4)=2X_zdz=(d z, -id z, -\frac{1}{z^2}d z, \frac{i}{z^2}d z), 
\end{align*}
where $i=\sqrt{-1}$. It is straightforward to see that meromorphic 1-forms in the data satisfy (\ref{EPWD}) in Proposition \ref{EPW}. This implies that the Lagrangian catenoid has 2 embedded planar ends in our sense. Moreover, the generalized Gauss map is given by $\Phi(z)=[1, -i, -\frac{1}{z^2}, \frac{i}{z^2}]\in\mathbb{P}^3$. Therefore the Gauss images of the two ends $z=0$ and $z=\infty$ are $[0, 0, 1, -i]\in\mathbb{P}^3$ and $[1, -i, 0, 0]\in\mathbb{P}^3$, respectively.}
\end{Ex}
As a direct consequence of the definition, we have the following characterization in genus 0:
\begin{prop}\label{DC}
A complete doubly-connected minimal immersion of finite total curvature in $\R^N$($N\geq4$) with embedded planar ends is contained in $\R^4(\simeq\C^2)\subseteq\R^N$, up to rigid motions. Moreover, it is given by a graph of $f(z)=az+\frac{b}{z}$ over $z\in\C-\{0\}$ for some $a\in\C$ and $b\in\C-\{0\}$.
\end{prop}
\begin{proof}
It has total curvature $-4\pi$ by the Gauss-Bonnet theorem. The result follows from the classification by Hoffman and Osserman \cite[Proposition~6.6]{HO1} and the zero-logarithmic-growth condition in our definition.
\end{proof}
In what follows, we introduce a special subclass of the embedded planar ends for 2-ended surfaces. First, assume that the ambient space is $\R^4$. The generalized Gauss map $\Phi: \Sigma_g\to\mathbb{P}^3$ factors through the quadric $Q\coloneqq\{[\zeta_1, \zeta_2, \zeta_3, \zeta_4]\in\mathbb{P}^3\ |\ \zeta_1^2+\zeta_2^2+\zeta_3^2+\zeta_4^2=0\}$. It is a ruled surface with two families of rulings:  
\begin{align}\label{ruling}
\begin{cases*}
L_{[a, b]}\coloneqq\{[ax+by, -i(ax-by), -bx+ay, -i(bx+ay)]\in Q\ |\ [x, y]\in\mathbb{P}^1\}\\
M_{[a, b]}\coloneqq\{[ax+by, -i(ax-by), bx-ay, -i(bx+ay)]\in Q\ |\ [x, y]\in\mathbb{P}^1\} 
\end{cases*}
\end{align}
for $[a, b]\in\mathbb{P}^1$. Now recall that for the Lagrangian catenoid, the Gauss images of the limit ends are $[1, -i, 0, 0]$, $[0, 0, 1, -i]\in\mathbb{P}^3$ (see Example \ref{EX1}). Here, the two Gauss images are on the same ruling $L_{[1, 0]}$.

Motivated from this observation, we say that a complete minimal immersion in $\R^4$ with 2 embedded planar ends has \emph{\textbf{ends like the Lagrangian catenoid}} if also the Gauss images of the limit ends are on the same ruling. Using the fact that rulings are given by the intersection of the quadric and tangent hyperplanes, we can also think of \emph{ends like the Lagrangian catenoid in $\R^N$} by Remark \ref{RVD}. Note that the Hoffman-Osserman surfaces that appeared in Proposition \ref{DC} are examples with ends like the Lagrangian catenoid.
\subsection{Kusner's theorem revisited}
In this subsection, we discuss $\R^4$-version of the Kusner theorem. We first recall the original theorem in \cite{K1}:
\begin{thm}[Theorem A in \cite{K1}]\label{KUS}
Let $M$ be a complete immersed minimal surface in $\R^3$ with finite total curvature and $\overline{M}$ be the one-point compactification of $M$ via the stereographic projection $\mathbb{S}^3\sim\R^3\cup\{\infty\}$. Denote by $n_x$ the multiplicity at a point $x\in\mathbb{S}^3$, and define $n(M)=\max\{n_x|x\in\R^3\}$ and $e(M)=n_{\infty}$. Then we have $n(M)\leq e(M)$ with the equality if and only if $M$ is a union of planes.
\end{thm}
Note that $e(M)$ is the number of ends provided that all ends are embedded. For the proof of the theorem, Kusner used the conformal invariance of $(|\vec{H}|^2-K)dA$ and the Li-Yau inequality for the Willmore functional \cite{LY} to show the inequality. Here $\vec{H}$ is the mean curvature vector and $K$ is the Gauss curvature. Moreover, the conformal invariance of mean curvature spheres (or equivalently, the conformal tangent bundle) was used to establish the equality case. See \cite{K1} for more details.

Now it is easy to see that the first two ingredients, the conformal invariance of $(|\vec{H}|^2-K)dA$ and the Li-Yau inequality, also work in arbitrary $\R^N$, $N\geq3$. Furthermore, there is also the conformal invariance of mean curvature spheres in dimension four. Indeed, conformal immersions of surfaces in $\R^4$ (or $\mathbb{S}^4$) can be interpreted in terms of the quaternions $\mathbb{H}$. According to this interpretation, the set of mean curvature spheres arises as a conformal map into the set of complex structures on $\mathbb{H}^2$. For detailed explanations, we may refer to \cite{Br1,Quat}.

By the observation above with the fact that the definition only deals with the graphical end, a slightly modified proof of \cite[Theorem~A]{K1} gives:
\begin{cor}\label{Kus4}
Let $M$ be a complete immersed minimal surface in $\R^4$ with finite total curvature and embedded planar ends. Then $n(M)\leq e(M)$, and the equality holds if and only if $M$ is a union of 2-planes. (Here, we used the same notation as in Theorem \ref{KUS}, replacing $\R^3$ with $\R^4$.)
\end{cor} 
%%%%%%%%%%%

\section{\textbf{Lemmas on the Weierstrass data}}\label{LEMM}
\setcounter{equation}{0}
Let $X: \Sigma_g\backslash\{p_1, p_2\}\to\R^4$ be a genus $g$($\geq1$) complete minimal immersion of finite total curvature with 2 embedded planar ends, which is represented by (\ref{WR}) with the Weierstrass data $(\phi_1, \phi_2, \phi_3, \phi_4)$. We assume that the Gauss image at $p_1$ is $[1, -i, 0, 0]\in\mathbb{P}^3$ by rotating the surface.

By the Riemann-Roch formula \cite{GH}, we have
\begin{align}\label{RR}
\begin{cases}
\dim_{\C} H^0(\Sigma_g, C)=g,\\ 
\dim_{\C} H^0(\Sigma_g, C\otimes[2p_1])=\dim_{\C} H^0(\Sigma_g, C\otimes[2p_2])=g+1,\\ 
\dim_{\C} H^0(\Sigma_g, C\otimes[2p_1+2p_2])=g+3,
\end{cases}
\end{align}
where $C$ denotes the canonical bundle of $\Sigma_g$. Therefore we may take
\begin{align}\label{eta}
\begin{cases}
\eta_1\in H^0(\Sigma_g, C\otimes[2p_1])\backslash H^0(\Sigma_g, C),\\
\eta_2\in H^0(\Sigma_g, C\otimes[2p_2])\backslash H^0(\Sigma_g, C).
\end{cases}
\end{align}

Since each $\eta_i$ has a pole only at $p_i$, the residue formula implies that
\begin{align}\label{residue}
\mbox{res}_{p_i}\eta_i=0,\ \mbox{ord}_{p_i}\eta_i=-2\ \forall i=1, 2.
\end{align}
Clearly, (\ref{eta}) also implies that $\mbox{res}_{p_1}\eta_2=\mbox{res}_{p_2}\eta_1=0$. 

\begin{lem}\label{ORF}
Let $\eta_1$ and $\eta_2$ be given as in (\ref{eta}). Then we have complex numbers $\alpha(\neq0)$, $a_g$, $b_g$, $c_g$, and $d_g$, where $a_g^2+b_g^2+c_g^2+d_g^2=0$ and
\begin{align}\label{GForm}
\begin{cases}
\phi_1=\omega_1+\alpha\eta_1+a_g\eta_2\\
\phi_2=\omega_2-i\alpha\eta_1+b_g\eta_2\\
\phi_3=\omega_3+c_g\eta_2\\
\phi_4=\omega_4+d_g\eta_2
\end{cases}
\end{align}
for some holomorphic 1-forms $\omega_i$'s on $\Sigma_g$. Note that $[a_g, b_g, c_g, d_g]\in\mathbb{P}^3$ represents the Gauss image at $p_2$.
\end{lem}

\begin{proof}
We see from (\ref{RR}) that there exists $\eta_3\in H^0(\Sigma_g, C\otimes[2p_1+2p_2])$ such that
\begin{align}\label{linear}
H^0(\Sigma_g, C\otimes[2p_1+2p_2])=H^0(\Sigma_g, C)\oplus\C\eta_1\oplus\C\eta_2\oplus\C\eta_3.
\end{align}
By the residue formula, we have 
\begin{align*}
\mbox{res}_{p_1}\eta_3+\mbox{res}_{p_2}\eta_3=0
\end{align*}
and this implies that 
\begin{align}\label{ress}
\mbox{res}_{p_1}\eta_3\neq0. 
\end{align}
Indeed, if $\mbox{res}_{p_1}\eta_3=0$, then $\mbox{res}_{p_2}\eta_3=0$ and one can conclude that $\eta_3\in H^0(\Sigma_g, C)\oplus\C\eta_1\oplus\C\eta_2$.

Now consider $\phi\in H^0(\Sigma_g, C\otimes[2p_1+2p_2])$ satisfying $\mbox{res}_{p_i}\phi=0$ $\forall i=1, 2$. By (\ref{linear}), we may write
\begin{align*}
\phi=\omega+r\eta_1+s\eta_2+t\eta_3,
\end{align*}
where $\omega\in H^0(\Sigma_g, C)$ and $r$, $s$, and $t$ are complex numbers. If we compute the residue at $p_1$, then (\ref{residue}) gives
\begin{align*}
0=\mbox{res}_{p_1}\phi=r\cdot\mbox{res}_{p_1}\eta_1+t\cdot\mbox{res}_{p_1}\eta_3=t\cdot\mbox{res}_{p_1}\eta_3
\end{align*}
and thus we obtain $t=0$ by (\ref{ress}). Therefore it follows from Proposition \ref{EPW} that $\phi_i$'s are contained in $H^0(\Sigma_g, C)\oplus\C\eta_1\oplus\C\eta_2$. 

Express each $\phi_j$ in a linear combination of $\eta_1$, $\eta_2$, and holomorphic 1-forms. As the Gauss image at each $p_i$ is obtained from the coefficients of $\eta_i$ in linear combinations, we obtain the expression in (\ref{GForm}).
\end{proof}
We need the following lemma for later use:
\begin{lem}\label{contra}
Suppose that the Weierstrass data is given in the form of (\ref{GForm}). Then $\phi_3^2+\phi_4^2\not\equiv0$.
\end{lem}
\begin{proof}
Assume that $\phi_3^2+\phi_4^2\equiv0$. By holomorphicity, this implies that $\phi_4\equiv i\phi_3$ or $\phi_4\equiv -i\phi_3$. Since $\phi_3$ and $\phi_4(=\pm i\phi_3)$ have no real periods, we may deduce that $\phi_3$ and $\phi_4$ have no periods. So we can integrate them to obtain well-defined meromorphic functions $\int\phi_3$ and $\int\phi_4$.

If we prove that $c_g\neq0$ or $d_g\neq0$, then at least one of the meromorphic functions above has a simple pole only at $p_2$. As such meromorphic function can only exist on the Riemann sphere, the proof is done. 

To prove $c_g\neq0$ or $d_g\neq0$, we assume the contrary: $c_g=d_g=0$. Then the two limit ends are parallel so that there is a hyperplane containing both asymptotic planes. Then the defining equation of the hyperplane gives a harmonic function on the surface, which vanishes at infinity. By the maximum principle, it must be identically zero on the surface. This shows that the surface is contained in the hyperplane. 

As there are no minimal immersions with 2 embedded planar ends in $\R^3$, this yields a contradiction and therefore $c_g\neq0$ or $d_g\neq0$.
\end{proof}

Now we specify $\eta_1$ and $\eta_2$ for each conformal structure:
\begin{lem}\label{ETAE}
If $g=1$, then $\Sigma_1$ and $p_1$ can be identified with $R_{\tau}$ and $0$, respectively, where $R_{\tau}\coloneqq\{x+y\tau\in\C\ |\ 0\leq x, y<1\}$ for some $\mbox{Im}(\tau)>0$. Let $z$ be the complex coordinate on $R_{\tau}$. Then $\eta_1$ and $\eta_2$ in (\ref{eta}) can be given as follows:
\begin{align*}
\eta_1=P(z)d z
\end{align*}
and
\begin{align*}
\eta_2=\begin{cases} 
\frac{P'(z)+P'(p_2)+\frac{P''(p_2)}{P'(p_2)}(P(z)-P(p_2))}{(P(z)-P(p_2))^2}d z\ &\text{if}\ \ p_2\not\in\{\frac{1}{2}, \frac{\tau}{2}, \frac{1+\tau}{2}\}\\
\frac{1}{P(z)-e_i}d z\ &\text{if}\ \ P(p_2)=e_i, i=1, 2, 3
\end{cases}.
\end{align*}
Here $P=P(z)$ is the Weierstrass $\mathcal{P}$-function on $R_{\tau}$ with the following identity:
\begin{align}\label{pfunc}
(P')^2=4(P-e_1)(P-e_2)(P-e_3),
\end{align} 
where $e_1\coloneqq P(\frac{1}{2})$, $e_2\coloneqq P(\frac{\tau}{2})$, and $e_3\coloneqq P(\frac{1+\tau}{2})$. 
\end{lem}

\begin{proof}
The first statement follows from the fact that the universal cover of $\Sigma_1$ is $\C$. Moreover, as $P$ is a meromorphic function with a double pole only at $z=0$, $P(z)d z$ satisfies (\ref{eta}). 

On the other hand, in the first case for $\eta_2$, $P(z)-P(p_2)$ has simple zeros at $z=p_2, -p_2$ and a double pole at $z=0$. The numerator $P'(z)+P'(p_2)+\frac{P''(p_2)}{P'(p_2)}(P(z)-P(p_2))$ has a double zero at $z=-p_2$, one simple zero, and an order 3 pole at $z=0$. Therefore the given meromorphic 1-form has a double pole only at $z=p_2$. A similar computation also applies to the second case, and the lemma follows.
\end{proof}

\begin{lem}\label{ETAH}
If $g\geq2$ and $\Sigma_g$ is hyperelliptic, then $\Sigma_g$ can be identified with $\{(x, y)\in\overline{\C}\times\overline{\C}\ |\ y^2=\prod_{j=1}^{2g+2}(x-\lambda_j)\}$ for some pairwise distinct complex numbers $\lambda_1, \cdots, \lambda_{2g+2}$. Moreover, a meromorphic 1-form $\eta_q$ with a double pole only at $q=(x_q, y_q)\in\Sigma_g$ can be given as follows:
\begin{align*}
\eta_q=\begin{cases}
\frac{y+y_q+\frac{d y}{d x}|_q(x-x_q)}{(x-x_q)^2}\frac{d x}{y}\ &\text{if}\ \ y_q\neq0\\
\frac{1}{x-x_q}\frac{d x}{y}\ &\text{if}\ \ y_q=0
\end{cases}.
\end{align*}
\end{lem}

\begin{proof}
For a general explanation on hyperelliptic Riemann surfaces, one may refer to \cite{GH}. As a similar case to the first one was covered in detail in Lemma \ref{ETAE}, so in this time, the second case is covered in detail. Note that also a similar computation below applies to the first case. 

Let $R(x)\coloneqq\prod_{j=1}^{2g+2}(x-\lambda_j)$. Away from the branch points $(\lambda_j, 0)$, we may write $y=\pm\sqrt{R(x)}$. There are two kinds of infinity as $x\to\infty$ that we will denote by $\infty_1$ and $\infty_2$. Then the divisor $(y)$ is given by $\sum_{j=1}^{2g+2}(\lambda_j, 0)-(g+1)\cdot\infty_1-(g+1)\cdot\infty_2$, and from $d x=\frac{2y}{R'}d y$, we obtain $(d x)=\sum_{j=1}^{2g+2}(\lambda_j, 0)-2\cdot\infty_1-2\cdot\infty_2$. Now we compute
\begin{align*}
\left(\frac{1}{x-\lambda_j}\frac{d x}{y}\right)&=(d x)-(y)-(x-\lambda_j)\\
&=g\cdot\infty_1+g\cdot\infty_2-2\cdot(\lambda_j, 0),
\end{align*}
where we used $(x-\lambda_j)=2\cdot(\lambda_j, 0)-\infty_1-\infty_2$. Therefore $\frac{1}{x-\lambda_j}\frac{d x}{y}$ has a double pole only at $(\lambda_j, 0)$ and it is an appropriate choice for $\eta_{(\lambda_j, 0)}$.
\end{proof}

\section{\textbf{Nonexistence}}\label{MAIN}
\setcounter{equation}{0}
In this section, we deal with nonexistence of complete minimal immersions in $\R^N$ with 2 embedded planar ends. Following Remark \ref{RVD}, it suffices to consider in some 4-dimensional Affine subspace and therefore we assume that $N=4$. Further, every theorem will be stated in $\R^4$, but the same results hold in general $\R^N$. 

We begin by observing the generalized Gauss map. To begin with, consider the map 
\begin{align*}
[\zeta_1, \zeta_2, \zeta_3, \zeta_4]\mapsto \left(\frac{\zeta_3+i\zeta_4}{\zeta_1-i\zeta_2}, \frac{-\zeta_3+i\zeta_4}{\zeta_1-i\zeta_2}\right)
\end{align*}
from $Q\backslash (L_{[1,0]}\cup M_{[1,0]})$ to $\C\times\C$, where $Q\subset\mathbb{P}^3$ is the quadric and $L_{[1,0]}$, $M_{[1,0]}$ are rulings defined in (\ref{ruling}). Since two rulings passing through the point $[\zeta_1, \zeta_2, \zeta_3, \zeta_4]\in Q\backslash(L_{[1, 0]}\cup M_{[1, 0]})$ are given by $L_{[\zeta_3+i\zeta_4, \zeta_1-i\zeta_2]}$ and $M_{[-\zeta_3+i\zeta_4, \zeta_1-i\zeta_2]}$, the above map naturally arises from the ruled structure of $Q$. Moreover, it extends to a biholomorphism from $Q$ onto $\mathbb{P}^1\times\mathbb{P}^1$. 

For each minimal immersion $X: \Sigma_g\backslash\{p_1, p_2\}\to\R^4$ with 2 embedded planar ends, let us denote $(g_1, g_2):\Sigma_g\to\mathbb{P}^1\times\mathbb{P}^1$ as the composition of the biholomorphism $Q\to \mathbb{P}^1\times\mathbb{P}^1$ and the generalized Gauss map $\Sigma_g\to Q\subset\mathbb{P}^3$. Note that we may write
\begin{align}\label{ggg}
g_1=\frac{\phi_3+i\phi_4}{\phi_1-i\phi_2},\ g_2=\frac{-\phi_3+i\phi_4}{\phi_1-i\phi_2},
\end{align}
when $\phi_1-i\phi_2\not\equiv0$.

Now recall that the Gauss curvature $K$ and the normal curvature $K^{\perp}$ satisfy
\begin{align}\label{jacobian}
K=J_{g_1}+J_{g_2},\ K^{\perp}=J_{g_1}-J_{g_2},
\end{align}
where $J_{g_1}$ and $J_{g_2}$ denote the Jacobian of $g_1$ and $g_2$, respectively (see \cite[Proposition~4.5]{HO2}). Let $S\coloneqq X(\Sigma_g\backslash\{p_1, p_2\})$. Then (\ref{jacobian}) and the Gauss-Bonnet formula imply that
\begin{align}\label{degp}
\deg g_1+\deg g_2=-\frac{1}{2\pi}\int_SKd A=2g+2,
\end{align}
where the sign comes from the orientation.

Next, we consider $\int_S K^{\perp}d A$. By the conformal invariance of $K^{\perp}d A$, we may regard $S$ as a surface in $\mathbb{S}^4$ via the stereographic projection. Denote by $\overline{S}$ the compactification of $S$, i.e., $\overline{S}=S\cup\{\infty\}$. It is clear from the definition of the embedded planar end that the normal bundle of $S$ extends to $\overline{S}$. 

Using the fact that $\frac{1}{2\pi}\int_{\overline{S}} K^{\perp}d A$ is twice the number of self-intersections, we get
\begin{align*}
\left|\frac{1}{2\pi}\int_{\overline{S}} K^{\perp}d A\right|=2
\end{align*}
since $S$ is embedded by Corollary \ref{Kus4} and $\infty$ is the only self-intersection point of $\overline{S}$. Hence (\ref{jacobian}) shows that
\begin{align}\label{degm}
|\deg g_1-\deg g_2|=\left|\frac{1}{2\pi}\int_{S} K^{\perp}d A\right|=\left|\frac{1}{2\pi}\int_{\overline{S}} K^{\perp}d A\right|=2.
\end{align}
Combining (\ref{degp}) and (\ref{degm}), we finally obtain
\begin{align}\label{tdeg}
(\deg g_1, \deg g_2)=(g, g+2)\ \text{or}\ (g+2, g).
\end{align}

We are now ready to prove the first main result.
\begin{thm}\label{G1}      
There are no genus 1 complete minimal immersions of finite total curvature in $\R^4$ with 2 embedded planar ends.
\end{thm}

\begin{proof}
Suppose the contrary. Let $X:\Sigma_1\backslash\{p_1, p_2\}\to\R^4$ be a genus 1 minimal immersion as in Section \ref{LEMM} so that the Gauss image at $p_1$ is $[1, -i, 0, 0]\in\mathbb{P}^3$. By Lemma \ref{ETAE}, we may consider $X: R_{\tau}\backslash\{0, p_2\}\to\R^4$, where $R_{\tau}=\{x+y\tau\in\C\ |\ 0\leq x, y<1\}$ with $\mbox{Im}(\tau)>0$. 

Since holomorphic 1-forms on $R_{\tau}$ are given by complex multiples of $d z$, it follows from (\ref{GForm}) that the Weierstrass data has the following form:
\begin{align*}
\begin{cases}
\phi_1=a_0d z+\alpha\eta_1+a_1\eta_2\\
\phi_2=b_0d z-i\alpha\eta_1+b_1\eta_2\\
\phi_3=c_0d z+c_1\eta_2\\
\phi_4=d_0d z+d_1\eta_2
\end{cases}
\end{align*}
such that $\alpha\neq0$ and $a_1^2+b_1^2+c_1^2+d_1^2=0$.

We may consider two cases depending on the value of $p_2$. In each case, we take $\eta_1$ and $\eta_2$ as in Lemma \ref{ETAE}. 

\underline{Case (i)}: Assume that $p_2\not\in\{\frac{1}{2}, \frac{\tau}{2}, \frac{1+\tau}{2}\}$. Let $\eta_1=P(z)d z$ and
\begin{align*}
\eta_2=\frac{P'(z)+P'(p_2)+\frac{P''(p_2)}{P'(p_2)}(P(z)-P(p_2))}{(P(z)-P(p_2))^2}d z(=:H(z)d z).
\end{align*}
Then $\sum_{j=1}^4\phi_j^2\equiv0$ if and only if
\begin{align*}
(a_0+\alpha P(z)+a_1H(z))^2&+(b_0-i\alpha P(z)+b_1H(z))^2\\
&+(c_0+c_1H(z))^2+(d_0+d_1H(z))^2\equiv0,
\end{align*}
which is equivalent to
\begin{align*}
&\left[(a_0^2+b_0^2+c_0^2+d_0^2)+2\alpha(a_0-ib_0)P(z)\right]\\
&+2\left[(a_0a_1+b_0b_1+c_0c_1+d_0d_1)+\alpha(a_1-ib_1)P(z)\right]H(z)\equiv0.
\end{align*}
Multiplying by $(P(z)-P(p_2))^2$ and separating the numerator of $H(z)$, we have
\begin{align*}
&-2\left[(a_0a_1+\cdots+d_0d_1)+\alpha(a_1-ib_1)P(z)\right]P'(z)\\
&\equiv\left[(a_0^2+\cdots+d_0^2)+2\alpha(a_0-ib_0)P(z)\right](P(z)-P(p_2))^2\\
&+2\left[(a_0a_1+\cdots+d_0d_1)+\alpha(a_1-ib_1)P(z)\right]\left(P'(p_2)+\frac{P''(p_2)}{P'(p_2)}(P(z)-P(p_2))\right).
\end{align*}
Squaring both sides and using the identity (\ref{pfunc}), we obtain
\begin{align*}
&16\left[(a_0a_1+\cdots+d_0d_1)+\alpha(a_1-ib_1)P(z)\right]^2(P(z)-e_1)(P(z)-e_2)(P(z)-e_3)\\
&\equiv\Bigg{[}\left[(a_0^2+\cdots+d_0^2)+2\alpha(a_0-ib_0)P(z)\right](P(z)-P(p_2))^2\\
&+2\left[(a_0a_1+\cdots+d_0d_1)+\alpha(a_1-ib_1)P(z)\right]\left(P'(p_2)+\frac{P''(p_2)}{P'(p_2)}(P(z)-P(p_2))\right)\Bigg{]}^2.
\end{align*}

We may consider each side as a polynomial of $P=P(z)$. As the Weierstrass $\mathcal{P}$-function takes infinitely many values, both sides must be the same polynomial. In the end, we get $a_0-ib_0=0$, $a_1-ib_1=0$, $c_0c_1+d_0d_1=0$, and $c_0^2+d_0^2=0$. Therefore $\phi_3^2+\phi_4^2\equiv0$. This is impossible by Lemma \ref{contra}.

\underline{Case (ii)}: Assume that $p_2\in\{\frac{1}{2}, \frac{\tau}{2}, \frac{1+\tau}{2}\}$. Let $\eta_1$ and $\eta_2$ be given by
\begin{align*}
\eta_1=P(z)d z,\ \eta_2=\frac{1}{P(z)-P(p_2)}d z.
\end{align*}
Consider the transformation $I$ on the universal cover $\C$ defined by $I(z)=-z$. Then it gives rise to an automorphism on $\Sigma_1$(or equivalently, $R_{\tau}$), which is also denoted by $I$ by abuse of notation. Since $P(z)$ is an even function, we see $I^*d z=-d z,\ I^*\eta_1=-\eta_1,\ \text{and}\ I^*\eta_2=-\eta_2$. Thus it follows that
\begin{align}\label{pullb}
I^*\phi_j=-\phi_j\ \forall j=1, 2, 3, 4.
\end{align}

Suppose that the immersion is given by
\begin{align*}
X(z)=\mbox{Re}\int_{z_0}^z(\phi_1, \phi_2, \phi_3, \phi_4)
\end{align*}
for some $z_0\in R_{\tau}\backslash\{0, p_2\}$ (see (\ref{WR})). We compute
\begin{align*}
X(I(z))&=\mbox{Re}\int_{z_0}^{I(z)}(\phi_1, \phi_2, \phi_3, \phi_4)\\
&=\mbox{Re}\int_{z_0}^{I(z_0)}(\phi_1, \phi_2, \phi_3, \phi_4)+\mbox{Re}\int_{I(z_0)}^{I(z)}(\phi_1, \phi_2, \phi_3, \phi_4)\\
&=V_{z_0}+\mbox{Re}\int_{z_0}^z(I^*\phi_1, I^*\phi_2, I^*\phi_3, I^*\phi_4)\\
&=V_{z_0}-\mbox{Re}\int_{z_0}^z(\phi_1, \phi_2, \phi_3, \phi_4)\\
&=V_{z_0}-X(z),
\end{align*}
where $V_{z_0}\coloneqq\int_{z_0}^{I(z_0)}(\phi_1, \phi_2, \phi_3, \phi_4)$, and we used (\ref{pullb}) in the fourth equality. Since the automorphism $I$ leaves each element of $\{0, \frac{1}{2}, \frac{\tau}{2}, \frac{1+\tau}{2}\}$ unchanged, we obtain
\begin{align*}
X(q)=\frac{1}{2}V_{z_0}\ \forall\ q\in\{\frac{1}{2}, \frac{\tau}{2}, \frac{1+\tau}{2}\}\backslash\{p_2\}.
\end{align*}
Therefore the image of $X$ has a point of multiplicity 2. By the equality case in Corollary \ref{Kus4}, it should be a union of planes, which is also a contradiction.
\end{proof}
Note that there is no assumption on the position of ends in the above theorem. One of the reasons for this is that when the genus is 1, there are plenty of conformal symmetries so that several cases could be excluded without calculating the periods. However, we only have finitely many automorphisms on a Riemann surface with the genus greater than or equal to 2. This makes the period computation necessary for many cases. 

We could overcome this difficulty by assuming that the conformal structure is hyperelliptic and by restricting the type of ends. Then we have the hyperelliptic involution so that a similar idea can be applied. 
\begin{thm}\label{G2}
If $g\neq4$, there are no genus $g$ hyperelliptic complete minimal immersions of finite total curvature in $\R^4$ with 2 embedded planar ends like the Lagrangian catenoid.
\end{thm}

\begin{proof}
Suppose that such minimal immersion exists. Let $X: \Sigma_g\backslash\{p_1, p_2\}\to\R^4$ be a genus g complete minimal immersion of finite total curvature with 2 embedded planar ends like the Lagrangian catenoid satisfying $g\neq4$. We assume that the Gauss image at $p_1$ is $[1, -i, 0, 0]\in\mathbb{P}^3$ as in Section \ref{LEMM}.

Following Lemma \ref{ETAH}, we from now on identify $\Sigma_g$ with 
\begin{align*}
\{(x, y)\in\overline{\C}\times\overline{\C}\ |\ y^2=R(x)\coloneqq\prod_{j=1}^{2g+2}(x-\lambda_j)\}
\end{align*} 
for some pairwise distinct complex numbers $\lambda_1, \cdots, \lambda_{2g+2}$. We may assume that $p_1, p_2\in\Sigma_g\backslash\{\infty_1, \infty_2\}$ by applying an automorphism on the first component $\overline{\C}$. On $\Sigma_g$, it is well-known that holomorphic 1-forms are generated by
\begin{align*}
\frac{d x}{y},\ x\frac{d x}{y},\ \cdots,\ x^{g-1}\frac{d x}{y}.
\end{align*}
Therefore, by (\ref{GForm}) with these holomorphic 1-forms, the Weierstrass data can be written as follows:
\begin{align*}
\begin{cases}
\phi_1=A(x)\frac{d x}{y}+\alpha\eta_1+a_g\eta_2\\
\phi_2=B(x)\frac{d x}{y}-i\alpha\eta_1+b_g\eta_2\\
\phi_3=C(x)\frac{d x}{y}+c_g\eta_2\\
\phi_4=D(x)\frac{d x}{y}+d_g\eta_2
\end{cases}
\end{align*}
such that $\alpha\neq0$ and $a_g^2+b_g^2+c_g^2+d_g^2=0$, and where $A(x)$, $B(x)$, $C(x)$, and $D(x)$ are polynomials of degree $\leq g-1$. Moreover, it should satisfy 
\begin{align}\label{Lagend}
a_g-ib_g=0 
\end{align}
to have the ends like the Lagrangian catenoid.

Now we consider three cases depending on the position of $p_1=(x_1, y_1)$ and $p_2=(x_2, y_2)$:
(i) $y_1\neq0$, $y_2\neq0$, (ii) $y_1\neq0$, $y_2=0$ or $y_1=0$, $y_2\neq0$, (iii) $y_1=0$, $y_2=0$.

\underline{Case (i)}: $y_1\neq0$, $y_2\neq0$. By Lemma \ref{ETAH} we may take
\begin{align*}
\eta_1=\frac{y+y_1+h_1(x-x_1)}{(x-x_1)^2}\frac{d x}{y},\ \eta_2=\frac{y+y_2+h_2(x-x_2)}{(x-x_2)^2}\frac{d x}{y},
\end{align*}
where $h_1\coloneqq\frac{d y}{d x}|_{p_1}$ and $h_2\coloneqq\frac{d y}{d x}|_{p_2}$. Then $\sum_{j=1}^4\phi_j^2\equiv0$ implies that
\begin{align*}
&\left(A(x)^2+B(x)^2+C(x)^2+D(x)^2\right)\\
&+2\left(a_gA(x)+b_gB(x)+c_gC(x)+d_gD(x)\right)\cdot\frac{y+y_2+h_2(x-x_2)}{(x-x_2)^2}\\
&+2\alpha\left(A(x)-iB(x)\right)\cdot\frac{y+y_1+h_1(x-x_1)}{(x-x_1)^2}\equiv0,
\end{align*}
where we used (\ref{Lagend}). 

Multiply by $(x-x_1)^2(x-x_2)^2$ and separate $y$ as we did in the proof of Theorem \ref{G1}. Then it follows that
\begin{align*}
-2&y\left[\left(a_gA(x)+b_gB(x)+c_gC(x)+d_gD(x)\right)(x-x_1)^2+\alpha\left(A(x)-iB(x)\right)(x-x_2)^2\right]\\
\equiv&\Big{[}\left(A(x)^2+B(x)^2+C(x)^2+D(x)^2\right)(x-x_1)^2(x-x_2)^2\\
&+2\left(a_gA(x)+b_gB(x)+c_gC(x)+d_gD(x)\right)(y_2+h_2(x-x_2))(x-x_1)^2\\
&+2\alpha\left(A(x)-iB(x)\right)(y_1+h_1(x-x_1))(x-x_2)^2\Big{]}
\end{align*}
and if we square both sides, we get
\begin{align*}
4&R(x)\left[\left(a_gA(x)+b_gB(x)+c_gC(x)+d_gD(x)\right)(x-x_1)^2+\alpha\left(A(x)-iB(x)\right)(x-x_2)^2\right]^2\\
\equiv&\Big{[}\left(A(x)^2+B(x)^2+C(x)^2+D(x)^2\right)(x-x_1)^2(x-x_2)^2\\
&+2\left(a_gA(x)+b_gB(x)+c_gC(x)+d_gD(x)\right)(y_2+h_2(x-x_2))(x-x_1)^2\\
&+2\alpha\left(A(x)-iB(x)\right)(y_1+h_1(x-x_1))(x-x_2)^2\Big{]}^2
\end{align*}
since $y^2=R(x)$. If we consider each side as a polynomial of $x$, then they must be the same polynomial. As $R(x)$ is square-free, we deduce that
\begin{align*}
\left(a_gA(x)+b_gB(x)+c_gC(x)+d_gD(x)\right)(x-x_1)^2+\alpha\left(A(x)-iB(x)\right)(x-x_2)^2\equiv0.
\end{align*}
By using (\ref{Lagend}), it can be expressed as
\begin{align}\label{eqeq}
\left(A(x)-iB(x)\right)(a_g(x-x_1)^2+\alpha(x-x_2)^2)+\left(c_gC(x)+d_gD(x)\right)(x-x_1)^2\equiv0.
\end{align}

If $A(x)-iB(x)\equiv0$, then 
\begin{align*}
\phi_1-i\phi_2\equiv(A(x)-iB(x))\frac{d x}{y}\equiv0,
\end{align*} 
which implies that $\phi_3^2+\phi_4^2\equiv0$. This is impossible by Lemma \ref{contra}. Therefore $A(x)-iB(x)\not\equiv0$. 

Since $(x-x_1)^2$ divides $A(x)-iB(x)$ in (\ref{eqeq}), we see from
\begin{align*}
2\leq \deg(A(x)-iB(x))\leq g-1 
\end{align*}
that $g\geq3$. Moreover, it follows from (\ref{eqeq}) that
\begin{align*}
\frac{c_g\phi_3+d_g\phi_4}{\phi_1-i\phi_2}=\frac{c_gC(x)+d_gD(x)}{A(x)-iB(x)}=-a_g-\alpha\cdot\frac{(x-x_2)^2}{(x-x_1)^2}
\end{align*}
is of degree 4. Therefore $g_1$ or $g_2$ is of degree 4 by (\ref{ggg}). Then (\ref{tdeg}) and $g\geq3$ imply that $g=4$, which contradicts to the assumption $g\neq4$. 

\underline{Case (ii)}: $y_1\neq0$, $y_2=0$ or $y_1=0$, $y_2\neq0$. First, assume that $y_1\neq0$ and $y_2=0$. Let
\begin{align*}
\eta_1=\frac{y+y_1+h_1(x-x_1)}{(x-x_1)^2}\frac{d x}{y},\ \eta_2=\frac{1}{x-x_2}\frac{d x}{y}
\end{align*}
as in Lemma \ref{ETAH}. By a similar computation used to obtain (\ref{eqeq}), we eventually get
\begin{align*}
A(x)-iB(x)\equiv0,
\end{align*}
and it follows that $\phi_1-i\phi_2\equiv0$. Therefore $\phi_3^2+\phi_4^2\equiv0$, which yields a contradiction by Lemma \ref{contra}. 

For the case $y_1=0$ and $y_2\neq0$, we similarly obtain 
\begin{align*}
a_gA(x)+b_gB(x)+c_gC(x)+d_gD(x)\equiv0
\end{align*} 
so that $a_g\phi_1+b_g\phi_2+c_g\phi_3+d_g\phi_4\equiv0$. This implies that the generalized Gauss map should lie in one of the two rulings passing through $[a_g, b_g, c_g, d_g]\in Q$. Since the Gauss image at $p_1$ is $[1, -i, 0, 0]\in Q$, we can conclude that $\phi_1-i\phi_2\equiv0$. Therefore we also have a contradiction.

\underline{Case (iii)}: $y_1=0$, $y_2=0$. Let
\begin{align*}
\eta_1=\frac{1}{x-x_1}\frac{d x}{y},\ \eta_2=\frac{1}{x-x_2}\frac{d x}{y}.
\end{align*}
Consider the hyperelliptic involution $I: \Sigma_g\to\Sigma_g$ defined by $I(x, y)=(x, -y)$. Then branch points are the fixed points of $I$. We also have $I^*\phi_j=-\phi_j$ $\forall j=1, 2, 3, 4$. Exactly the same computation in case (ii) of the proof of Theorem \ref{G1} shows that the image of $X$ has a point of multiplicity at least $2g$. It contradicts Corollary \ref{Kus4}, and the theorem is proved.
\end{proof}
\begin{rmk}
\textup{It is still an open problem for $g=4$ in Theorem \ref{G2}. Case (ii) and (iii) in the proof are true even in $g=4$, but case (i) requires period computations. On the other hand, for non-hyperelliptic cases, we also face period problems. For instance, if $g=3$, then it has an advantage that the canonical curve is given by a plane quartic curve. However, one can check by an explicit computation that only applying the algebraic computation is not sufficient to rule out all possible cases.}
\end{rmk}
%%%%%%%%%%%%%%

\section{\textbf{Examples with $d\geq3$ embedded planar ends in $\R^4$}}\label{MAINN}
\setcounter{equation}{0}
In this section, we construct some minimal surfaces in $\R^4$ with more than 2 embedded planar ends. All examples are given by holomorphic curves in $\C^2$. Some of them correspond to cases where the nonexistence was proved in $\R^3$. 

\begin{prop}
Let $F: \C\backslash\{z_1, z_2\}\to\C^2$ be the map defined by
\begin{align*}
F(z)=\left(-\frac{a_1}{z-z_1}-\frac{a_2}{z-z_2}+a_3z, -\frac{b_1}{z-z_1}-\frac{b_2}{z-z_2}+b_3z\right)
\end{align*}
for some $z_1, z_2\in\C\backslash\{0\}$. Suppose that $a_1b_2-b_1a_2=0$ and 
\begin{align*} 
\begin{pmatrix}
a_1 & a_2 & a_3\\
b_1 & b_2 & b_3
\end{pmatrix}
\end{align*}
is of rank 2. Then $F$ represents a genus 0 complete minimal embedding of finite total curvature with 3 embedded planar ends. Moreover, two of the three ends are parallel.
\end{prop}
\begin{proof}
Since it is a holomorphic curve, it becomes a branched minimal immersion in $\R^4$. The Weierstrass data is given by
\begin{align*}
(\phi_1, \phi_2, \phi_3, \phi_4)=(\phi_a, -i\phi_a, \phi_b, -i\phi_b),
\end{align*}
where
\begin{align*}
\phi_a:=\left(\frac{a_1}{(z-z_1)^2}+\frac{a_2}{(z-z_2)^2}+a_3\right)d z,\ \phi_b:=\left(\frac{b_1}{(z-z_1)^2}+\frac{b_2}{(z-z_2)^2}+b_3\right)d z.
\end{align*}
Hence it has embedded planar ends at $z=z_1$, $z_2$, and $\infty$. The Gauss images of the limit ends are $[a_k, -ia_k, b_k, -ib_k]\in\mathbb{P}^3$ ($k=1, 2, 3$). Therefore $a_1b_2-b_1a_2=0$ shows that two of the three ends are parallel.

On the other hand, it is clear that the kernel has the following form:
\begin{align*}
\ker \begin{pmatrix}
a_1 & a_2 & a_3\\
b_1 & b_2 & b_3
\end{pmatrix}=\C
\begin{pmatrix}
*\\
*\\
0
\end{pmatrix}.
\end{align*}
This proves that $\phi_a$ and $\phi_b$ have no common zeros, so $F$ is an immersion. It also implies that $F$ is embedded, and the proposition is proved.
\end{proof}
\begin{rmk}
\textup{Applying \cite[Proposition~6.5]{HO1}, one can show that the converse is also true; up to rigid motions, a genus 0 complete minimal embedding of finite total curvature with 3 embedded planar ends is given by $F$ for some $a_k$'s and $b_k$'s satisfying the same conditions. But we will not proceed with any details here.}
\end{rmk}

Next, we consider a positive genus case. Let $\Sigma_g=\{(x, y)\in\overline{\C}\times\overline{\C}\ |\ y^2=\prod_{j=1}^{2g+2}(x-\lambda_j)\}$ be a Riemann surface of genus $g$, where $\lambda_j$'s are all distinct. Note that it is elliptic if $g=1$ and hyperelliptic if $g\geq2$. For each nonempty subset $I\subseteq\{1, 2, \cdots, 2g+2\}$, we define the meromorphic function
\begin{align*}
f_I\coloneqq\frac{y}{\prod_{i\in I}(x-\lambda_i)}
\end{align*}
on $\Sigma_g$. As a divisor, we may write 
\begin{align}\label{fidiv}
(f_I)=\sum_{k\in I^c}(\lambda_k, 0)+(|I|-g-1)\cdot\infty_1+(|I|-g-1)\cdot\infty_2-\sum_{i\in I}(\lambda_i, 0).
\end{align}
%RRRRRr(\textcolor{red}{from RR and RR proof reason for expression})
\begin{lem}\label{Exa}
Let $|I|\geq g+1$. Then the meromorphic 1-form $d f_I$ satisfies the following:
\begin{itemize}
\item[(1)] $d f_I$ has double poles at $(\lambda_i, 0)$ $\forall i\in I$.
\item[(2)] $d f_I$ has no zeros at $(\lambda_j, 0)$, $1\leq j\leq 2g+2$.
\item[(3)] If $|I|\geq g+3$, then $df_I$ has a zero of order $(|I|-g-2)$ at each $\infty_1$, $\infty_2$. Otherwise, it has no zeros at $\infty_1$, $\infty_2$.
\item[(4)] $\emph{res}_{(\lambda_i, 0)}df_I=0$, $\forall i\in I$.
\end{itemize}
\end{lem}

\begin{proof}
(1), (2), and (3) are clear from (\ref{fidiv}). To prove (4), we compute
\begin{align}\label{diff}
d f_I&=\frac{1}{\prod_{i\in I}(x-\lambda_i)}d y-\frac{y}{\left(\prod_{i\in I}(x-\lambda_i)\right)^2}d \left(\prod_{i\in I}(x-\lambda_i)\right)\nonumber\\
&=\frac{1}{\prod_{i\in I}(x-\lambda_i)}d y-\frac{y}{\prod_{i\in I}(x-\lambda_i)}\left(\sum_{i\in I}\frac{1}{x-\lambda_i}\right)d x.
\end{align}
Taking the derivative in $y^2=\prod_{j=1}^{2g+2}(x-\lambda_j)$, this gives
\begin{align*}
2yd y=\left(\prod_{j=1}^{2g+2}(x-\lambda_j)\right)\left(\sum_{j=1}^{2g+2}\frac{1}{x-\lambda_j}\right)d x=y^2\left(\sum_{j=1}^{2g+2}\frac{1}{x-\lambda_j}\right)d x
\end{align*}
and therefore
\begin{align}\label{difff}
\frac{2d y}{y^2}=\left(\sum_{j=1}^{2g+2}\frac{1}{x-\lambda_j}\right)\frac{d x}{y}.
\end{align}
Hence, from (\ref{diff}) and (\ref{difff}), we obtain
\begin{align}\label{diffff}
d f_I&=\frac{1}{2}\frac{y^2}{\prod_{i\in I}(x-\lambda_i)}\left(\sum_{j=1}^{2g+2}\frac{1}{x-\lambda_j}-2\sum_{i\in I}\frac{1}{x-\lambda_i}\right)\frac{d x}{y}\nonumber\\
&=\frac{1}{2}\frac{y^2}{\prod_{i\in I}(x-\lambda_i)}\left(\sum_{j\in I^c}\frac{1}{x-\lambda_j}-\sum_{i\in I}\frac{1}{x-\lambda_i}\right)\frac{d x}{y}\nonumber\\
&=\frac{1}{2}\left(\prod_{j\in I^c}(x-\lambda_j)\right)\left(\sum_{j\in I^c}\frac{1}{x-\lambda_j}-\sum_{i\in I}\frac{1}{x-\lambda_i}\right)\frac{d x}{y},
\end{align}
where $y^2=\prod_{j=1}^{2g+2}(x-\lambda_j)$ was used in the last equality.

For each $s\in\{1, 2, \cdots, 2g+2\}$, define
\begin{align}\label{coef}
\sigma_{I; s}\coloneqq\prod_{j\in I^c}(\lambda_s-\lambda_j)
\end{align}
and
\begin{align*}
h_{I; s}(x)\coloneqq\frac{1}{x-\lambda_s}\left(\prod_{j\in I^c}(x-\lambda_j)-\sigma_{I; s}\right).
\end{align*}
Moreover, let us denote
\begin{align*}
H_I(x)\coloneqq\sum_{s\in I^c}h_{I; s}(x)=\left(\prod_{j\in I^c}(x-\lambda_j)\right)\left(\sum_{j\in I^c}\frac{1}{x-\lambda_j}\right).
\end{align*}
Then we may write (\ref{diffff}) as
\begin{align}\label{df}
d f_I=\frac{1}{2}\left(H_{I}(x)-\sum_{i\in I}h_{I; i}(x)-\sum_{i\in I}\frac{\sigma_{I; i}}{x-\lambda_i}\right)\frac{d x}{y}.
\end{align}

Consider $H_I(x)$ and $h_{I; i}(x)$ $(i\in I)$ as polynomials of $x$. If $|I|\geq g+1$, then 
\begin{align*}
\deg\left(H_{I}(x)-\sum_{i\in I}h_{I; i}(x)\right)\leq g-1,
\end{align*}
and we see that
\begin{align*}
\frac{1}{2}\left(H_{I}(x)-\sum_{i\in I}h_{I; i}(x)\right)\frac{d x}{y}
\end{align*}
defines a holomorphic 1-form on $\Sigma_g$. Now (4) is proved by Lemma \ref{ETAH}.
\end{proof}

Let $I, J\subseteq \{1, 2, \cdots, 2g+2\}$ be two subsets satisfying $I\neq J$ and $|I|, |J|\geq g+1$. To construct examples with embedded planar ends, we consider the map $\Phi_{I, J}$ given by
\begin{align}\label{exmap}
\Phi_{I, J}=(f_I, f_J)
\end{align}
from $\Sigma_g\backslash\{(\lambda_k, 0)\in\Sigma_g|k\in I\cup J\}$ to $\C^2$. It becomes a branched minimal immersion in $\R^4$ with the Weierstrass data given as
\begin{align*}
(\phi_1, \phi_2, \phi_3, \phi_4)=(d f_I, -id f_I, d f_J, -id f_J),
\end{align*}
where branch points can occur at common zeros of $d f_I$ and $d f_J$. Now it is an easy consequence of Lemma \ref{Exa} that the data satisfies (\ref{EPWD}) in Proposition \ref{EPW}. Therefore we obtain a branched minimal immersion in $\R^4$ with $|I\cup J|$ embedded planar ends. 

In the following, we will show that branch points cannot exist for some special choices of $I$ and $J$. We from now on assume that $|I|\geq |J|$. We also assume that $|J|\leq g+2$ to rule out the possible branch points at $\infty_1$ and $\infty_2$ (see (3) in Lemma \ref{Exa}).
\begin{lem}\label{EM}
Suppose that $I$, $J$ satisfy one of the following conditions:
\begin{itemize}
\item[(1)] $|I|=|J|$ and $|I\cap J|=g$ or $g+1$.
\item[(2)] $J\subset I$ and $|I|=g+2$, $|J|=g+1$.
\end{itemize}
Then $d f_I$ and $d f_J$ have no common zeros.
\end{lem}
\begin{proof}
Recall from (\ref{diffff}) that 
\begin{align*}
d f_I=\frac{1}{2}\left(\prod_{k\in I^c}(x-\lambda_k)\right)\left(\sum_{k\in I^c}\frac{1}{x-\lambda_k}-\sum_{i\in I}\frac{1}{x-\lambda_i}\right)\frac{d x}{y}
\end{align*}
and
\begin{align*}
d f_J=\frac{1}{2}\left(\prod_{k\in J^c}(x-\lambda_k)\right)\left(\sum_{k\in J^c}\frac{1}{x-\lambda_k}-\sum_{j\in J}\frac{1}{x-\lambda_j}\right)\frac{d x}{y}.
\end{align*}
If $I$, $J$ satisfy (1) or (2), then common zeros of 
\begin{align*}
\sum_{k\in I^c}\frac{1}{x-\lambda_k}-\sum_{i\in I}\frac{1}{x-\lambda_i},\ \sum_{k\in J^c}\frac{1}{x-\lambda_k}-\sum_{j\in J}\frac{1}{x-\lambda_j}
\end{align*}
must belong to $\{\infty_1, \infty_2\}$. Indeed,
\begin{align*}
&\left(\sum_{k\in I^c}\frac{1}{x-\lambda_k}-\sum_{i\in I}\frac{1}{x-\lambda_i}\right)-\left(\sum_{k\in J^c}\frac{1}{x-\lambda_k}-\sum_{j\in J}\frac{1}{x-\lambda_j}\right)\\
&=2\left(\sum_{k\in J\backslash I}\frac{1}{x-\lambda_k}-\sum_{k\in I\backslash J}\frac{1}{x-\lambda_k}\right)
\end{align*}
and since both (1) and (2) imply that $|I\backslash J|\leq 1$ and $|J\backslash I|\leq 1$, so only possible zeros are $\infty_1$ and $\infty_2$. Hence, by Lemma \ref{Exa}, we can finish the proof.
\end{proof}
It should be noted that there is no assumption on $\lambda_k$'s in the previous lemma. However, choosing $\lambda_k$'s carefully can handle more cases:
\begin{lem}\label{IMM}
If $|I\cup J|\leq 2g+1$, then $\lambda_k$'s can be chosen so that $d f_I$ and $d f_J$ have no common zeros.
\end{lem}
\begin{proof}
Suppose that $\lambda_k$ $\forall k\in I\cup J$ are given. We may express
\begin{align*}
&\sum_{k\in I^c}\frac{1}{x-\lambda_k}-\sum_{i\in I}\frac{1}{x-\lambda_i}\\
&=\left(\sum_{k\in (I\cup J)^c}\frac{1}{x-\lambda_k}-\sum_{k\in I\cap J}\frac{1}{x-\lambda_k}\right)+\left(\sum_{k\in J\backslash I}\frac{1}{x-\lambda_k}-\sum_{k\in I\backslash J}\frac{1}{x-\lambda_k}\right)
\end{align*}
and choose $\lambda_k$ $\forall k\in (I\cup J)^c$ so that
\begin{align*}
\sum_{k\in (I\cup J)^c}\frac{1}{x-\lambda_k}-\sum_{k\in I\cap J}\frac{1}{x-\lambda_k}
\end{align*}
does not vanish on the zeros of 
\begin{align*}
\sum_{k\in J\backslash I}\frac{1}{x-\lambda_k}-\sum_{k\in I\backslash J}\frac{1}{x-\lambda_k}.
\end{align*}
It is possible since there are only finitely many zeros. 

For such $\lambda_k$'s, one can show that
\begin{align*}
\sum_{k\in I^c}\frac{1}{x-\lambda_k}-\sum_{i\in I}\frac{1}{x-\lambda_i},\ \sum_{k\in J^c}\frac{1}{x-\lambda_k}-\sum_{j\in J}\frac{1}{x-\lambda_j}
\end{align*}
cannot have common zeros. Then the lemma can be proved by Lemma \ref{Exa}.
\end{proof}
Combining Lemma \ref{EM} and \ref{IMM}, we have proved that genus $g$ minimal immersions with $|I\cup J|$ embedded planar ends can be given in the form of (\ref{exmap}). All immersions are symmetric with respect to the origin because of the hyperelliptic involution. Moreover, if $g=1$, then one can show easily that all 3-ended minimal immersions coming from Lemma \ref{EM} are embedded. Therefore we have the following theorem:

\begin{thm}
Let $g\geq 1$ and $g+2\leq d\leq 2g+1$. For every $(g, d)$, there exists a complete minimal immersion of finite total curvature in $\R^4$ with $d$ embedded planar ends and genus $g$, that is symmetric with respect to the origin and hyperelliptic if $g\geq2$. Moreover, there are embedded examples when $g=1$ and $d=3$.
\end{thm}
By (\ref{coef}) and (\ref{df}) in the proof of Lemma \ref{Exa}, the Gauss images of the limit ends are given by 
\begin{align*}
[\sigma_{I; s}, -i\sigma_{I; s}, \sigma_{J; s}, -i\sigma_{J; s}]\in\mathbb{P}^3,\ s\in I\cup J.
\end{align*}
Since Lemma \ref{EM} is true for arbitrary $\lambda_k$'s, we are able to obtain minimal immersions with various kinds of end positions. So, in particular, the above theorem provides a family of embedded minimal tori in $\R^4$ with 3 embedded planar ends. We would like to emphasize that these examples are of particular interest as there are no such minimal immersions in $\R^3$.

%%%%%%%%%%%%%%%%%%%%%%%%%%%%%%%%%%%%%%%%%%
	
%%%%%%%%%%%%%%%%%%%%%%%%%%%%%%%%%%%%%%%%%%%%%%%%%

\end{document}